\documentclass[11pt]{article}

\usepackage[english]{babel}
\usepackage[utf8]{inputenc}
\usepackage{amsmath}
\usepackage{amssymb}
\usepackage{amsthm}
\usepackage{enumerate}
\usepackage{graphicx}
\usepackage{fullpage}

\newtheorem{theorem}{Theorem}[section]

\newtheorem{corollary}[theorem]{Corollary}
\newtheorem{proposition}[theorem]{Proposition}

\renewcommand{\a}{\alpha}
\newcommand{\ka}{\kappa}
\newcommand{\lmin}{\lambda_{\rm min}}
\newcommand{\lmax}{\lambda_{\rm max}}
\newcommand{\RR}{\mathbb{R}}
\newcommand{\CC}{\mathbb{C}}

\newcommand{\baromega}{\overline{\omega}}
\newcommand\sumvtou[1]{\sum_{v\to u} #1}

\title{A new kind of Hermitian matrices for digraphs}

\author{Bojan Mohar%
  \thanks{The author was supported in part by the NSERC Discovery Grant R611450 (Canada), by the Canada Research Chairs program, and by the Research Project J1-8130 of ARRS (Slovenia).}~%
  \thanks{On leave from IMFM, Department of Mathematics, University of Ljubljana.}\\
  Department of Mathematics\\
  Simon Fraser University\\
  \texttt{mohar@sfu.ca}
}

\date{\today}

\begin{document}

\maketitle

\begin{abstract}
In an earlier work, the author together with Guo \cite{GuMo} introduced the Hermitian adjacency matrix of directed (and partially directed) graphs. However, it appears that a more natural Hermitian matrix exists, and it is the purpose of this note to bring this new Hermitian matrix to the attention of researchers in algebraic graph theory.
\end{abstract}

\section{Introduction}
\label{sec:intro}

Eigenvalues of graphs have diverse applications in combinatorics. We refer to \cite{Bi93, CDS95, GR, BH, Sp12} and references therein.
On the other hand, results about eigenvalues of digraphs are sparse. One reason is that it is not clear which matrix associated to a digraph $D$ would best reflect interesting combinatorial properties in its spectrum. One candidate is the \emph{adjacency matrix} $A=A(D)$ whose $(u,v)$-entry is 1 if there is an arc from the vertex $u$ to $v$, and 0 otherwise. A well-known theorem of Wilf \cite{Wilf} bounding the chromatic number in terms of its largest eigenvalue extends to this setting as shown in \cite{M10}. However, the disadvantage of this matrix is that it is not symmetric and we lose the property that eigenvalues are real. Moreover, algebraic and geometric multiplicities of eigenvalues may be different. Another candidate is the \emph{skew-symmetric adjacency matrix} $S(D)$, where the $(u,v)$-entry is 1 if there is an arc from $u$ to $v$, and $-1$ if there is an arc from $v$ to $u$ (and $0$ otherwise). This choice is quite natural but it only works for \emph{oriented graphs} (i.e.\ when we have no digons). We refer to a survey by Cavers et al.\ \cite{CaCietal12}.

A Hermitian adjacency matrix for digraphs and mixed graphs was introduced recently by Guo and Mohar \cite{GuMo}, who provided many basic properties of this matrix. The same matrix was independently used by Liu and Li \cite{LiuLi15}, who considered an application in mathematical chemistry. The $uv$-entry of this Hermitian matrix is equal to 1 if there is an unoriented edge between $u$ and $v$, it is equal to $i=\sqrt{-1}$ if $uv$ is an arc and is equal to $-i$ if $vu$ is an arc. Of course, one may ask whether taking complex numbers $i$ and $-i$ is the most natural choice. In this note we offer another choice and argue why that one may be more natural to be used when investigating relationship between eigenvalues and combinatorial properties of digraphs and, more generally, of mixed multigraphs.

With this new Hermitian matrix, each arc directed from $v$ to $u$ contributes the sixth root of unity $\omega = \bigl(1+i\sqrt{3}\,\bigr)/2$ to the $vu$-entry in the matrix and contributes $\baromega = \bigl(1-i\sqrt{3}\,\bigr)/2$ to the $uv$-entry. Note that two oppositely directed arcs between $u$ and $v$ together contribute 1 to each of the entries. In this way, digons have the same effect on the adjacency matrix as undirected edges.

The main reason why the sixth root of unity is natural in relation to combinatorial properties is that $\omega \cdot \overline{\omega} = 1$ and $\omega + \overline{\omega} = 1$. The product of ``weights'' of edges is natural with the counting of closed walks (which is related to the entries of the powers of the adjacency matrix). The sum is needed when dealing with multiple arcs or with weighted edges.

The sixth root of unity appears naturally across applications. It appears in the definition Eisenstein integers; in relation to abstract linear independence (matroids), the sixth root matroids play a special role next to regular and binary matroids \cite{OVW98,Wh97}. They also seem to have connections to Quantum Field Theory \cite{KaKn10}.

In this paper we introduce this new Hermitian matrix and prove a few basic results about its eigenvalues.

\section{Basic properties}
\label{sec:defn}

\subsection{Digraphs and their Hermitian matrices}

A \emph{directed graph} (or \emph{digraph}) $D$ consists of a finite set $V=V(D)$ of vertices together with a set $E=E(D)$ of \emph{arcs} or \emph{directed edges}. Each arc $e\in E$ joins two vertices $u,v\in V$ and one of these, say $v$, is its \emph{initial vertex}, while the other one, $u$, is its \emph{terminal vertex}. To denote this incidence we write $e=vu$ for short. If $vu \in E$ and $uv \in E$, we say that the pair $\{uv, vu\}$ of these oppositely directed arcs is a \emph{digon} of $D$.

Multiple arcs and multiple edges between the same pair of vertices are allowed. Although we sometimes write $uv\in E(D)$, we mean by this any arc from $u$ to $v$. Although we will not discuss loops, they may be present, but they should be considered as undirected (or as ``two oppositely oriented" loops at the same vertex).

Let $\a = a+bi$ be a complex number with absolute value 1, $|\a|=1$, where $a\ge0$, and let $\bar\a = a-bi$ be its conjugate. For a digraph $D$ with vertex set $V = V(D)$ and arc set $E = E(D)$, we consider the \emph{Hermitian adjacency matrix} $N^\a = N^\a(D) \in \CC^{V\times V}$, whose entries $N_{uv}^\a$ are given by
\[
N_{uv}^\a = e(u,v)\,\a + e(v,u)\, \bar\a,
\]
where $e(x,y)$ denotes the number of arcs from $x$ to $y$. Observe that $N_{uv}^\a$ and $N_{vu}^\a$ are conjugate to each other, and therefore $N^\a$ is a Hermitian matrix. If $D$ has loops, each loop contributes $\a+\bar\a = 2a$ to the corresponding diagonal entry.

\subsection{About the most natural choice of $\a$}

Let
$$
   \omega = \frac{1+i\sqrt{3}}{2}
$$
be the primitive sixth root of unity and let $\baromega = \tfrac{1-i\sqrt{3}}{2}$ be its conjugate.

For a digraph $D$, we consider the corresponding Hermitian adjacency matrix $M(D) = N^\omega(D)$ and we refer to it as the \emph{Hermitian matrix of the second kind}.

If every edge of $D$ lies in a digon, then $M(D) = A(D)$, because $\omega + \baromega=1$. This reflects that $D$ is, essentially, equivalent to an undirected graph in such a case. More generally, a mixed graph is a graph where directed and undirected edges may coexist. Formally, a \emph{mixed graph} is an ordered triple $(V, E, A)$ where $V$ is the vertex-set, $E$ is a set of undirected edges, and $A$ is set of arcs, or directed edges. The Hermitian adjacency matrix of the second kind is defined in such a way that all undirected edges may be replaced by digons and, from this perspective, mixed graphs are equivalent to the class of digraphs that we consider here.

The number $\omega$ used in defining the matrix $M(D)$ satisfies the following identities:
$$
    \omega \cdot \baromega = 1 \quad \textrm{and} \quad \omega + \baromega = 1.
$$
While the first condition holds for any $\a$ of absolute value 1, the second identity gives an important \emph{additivity} property and enables us to view oppositely oriented arcs as an unoriented edge and works naturally with multiple arcs and multiple edges. The first one is multiplicative and plays the role in expressing powers of the matrix $M^k$ which correspond, in the same way as the powers of the usual adjacency matrices, to counting ``walks'' of length $k$, the second one allows this property to extend when multiple edges are present, and also works in the setting where the edges have arbitrary positive weights.

Besides the additivity property being implied by our choice of the sixth root of unity, there are several other reasons why this choice may be the most natural among all possible choices for the entries of Hermitian adjacency matrices. The sixth root of unity appears naturally across applications. One such instance is that the sixth-root matroids play a special role in formal theory of linear independence, see, for example, \cite{OVW98,Wh97,Re09,CDN15}. Another natural setup comes from theoretical physics, see \cite{KaKn10} for an example.

\subsection{Eigenvalues}

Observe that $N^\a(D)$ is a Hermitian matrix and so is diagonalizable with real eigenvalues. The following proposition contains properties that are true for adjacency matrices which also carry over to the Hermitian case.

\begin{proposition}\label{prop:real}
Every Hermitian adjacency matrix $N^\a = N^\a(D)$ of a digraph $D$ with vertex-set $V$ has the following properties:
\begin{enumerate}[\rm (i)]
\item All eigenvalues of $N^\a$ are real.
\item The matrix $N^\a$ has $|V|$ pairwise orthogonal eigenvectors in $\CC^V$ and is unitarily similar to a diagonal matrix.
\item The numerical range of $N^\a$, defined as the set $R = \{z^*N^\a z \mid z\in \CC^V, \Vert z \Vert = 1\}$ is an interval of real numbers and $\min R = \lmin(N^\a)$ is the smallest eigenvalue of $N^\a$ and $\max R = \lmax(N^\a)$ is the largest eigenvalue of $N^\a$.
\end{enumerate}
\end{proposition}

The eigenvalues of $N^\a(D)$ are the \emph{$N^\a$-eigenvalues of $D$} and the spectrum of $N^\a(D)$ (i.e.\ the multiset of eigenvalues, counting their multiplicities) is the \emph{$N^\a$-spectrum of $D$}. Similar terminology is used for the particular case of the Hermitian matrix $M(D)$  of the second kind. The $N^\a$-eigenvalues of a digraph $D$ will be ordered in the decreasing order, the $j$th largest eigenvalue will be denoted by $\nu_j^{(\a)}(D)$, so that $\nu_1^{(\a)}(D)\ge \nu_2^{(\a)}(D)\ge \cdots \ge \nu_n^{(\a)}(D)$ ($n=|V|$). For the special case of the matrix $M(D)$, the eigenvalues are $\mu_1(D)\ge \mu_2(D)\ge \cdots \ge \mu_n(D)$.

A direct consequence of Proposition \ref{prop:real} is the min-max formula for $\nu_j^{(\a)}(D)$:
\begin{equation}
  \nu_j^{(\a)}(D) ~= \max_{\dim U=j} ~ \min_{\substack{z\in U\\ \Vert z \Vert = 1}} ~ z^*N^\a z ~= \min_{\dim U=n-j+1} ~ \max_{\substack{z\in U\\ \Vert z \Vert = 1}} ~ z^*N^\a z
\label{eq:minmax jth largest}
\end{equation}
where the outer maximum (minimum) is taken over all subspaces $U$ of $\CC^V$ of dimension $j$ ($n-j+1$, respectively) and the inner minimum (maximum) is taken over all unit vectors $z\in U$.

For $z\in \CC^V$, the following expression of the quadratic form $z^*N^\a z$ shows how individual arcs contribute to it. In this expression we use the summation over all arcs $vu$, which we indicate as $\sumvtou{f(v,u)}$. The sum runs over all arcs $e=vu$, where each arc is considered as many times as its multiplicity. Let $z=x+iy$, where $x,y\in \RR^V$. With $\a = a+bi$ we have:
\begin{eqnarray}
   z^*N^\a z & = & \sum_{v\in V} \overline{z_v} \sum_{u\in V} N_{vu}^\a z_u \nonumber \\
         & = & \sum_{v\in V} (x_v-iy_v) \sum_{u\in V} N_{vu}^\a (x_u+iy_u) \nonumber \\
         & = & \sumvtou{((x_v-iy_v)(x_u+iy_u)\,\a+(x_u-iy_u)(x_v+iy_v)\,\bar\a)} \nonumber \\
         & = & \sumvtou{(2a\,x_vx_u + 2a\,y_vy_u - 2b\,x_vy_u + 2b\,y_vx_u)}.
\label{eq:zNz}
\end{eqnarray}

This expression has a clear combinatorial meaning when combined together with the min-max formula (\ref{eq:minmax jth largest}).

Another property from linear algebra that is used in many combinatorial applications of graph eigenvalues is interlacing. Since $N^\a(D)$ is Hermitian, it has the interlacing property which we describe next.

Suppose that $\nu_1\ge \nu_2\ge \cdots \ge \nu_n$ and $\ka_1\ge \ka_2\ge \cdots \ge \ka_{n-t}$ (where $t\ge1$ is an integer) are sequences of real numbers. We say that the sequences $\nu_l$ ($1\le l\le n$) and $\ka_j$ ($1\le j\le n-t$) \emph{interlace} if for every $s=1,\dots,n-t$, we have
$$
    \nu_s\ge \ka_s\ge \nu_{s+t}.
$$

The usual version of the eigenvalue interlacing property states that the eigenvalues of any principal submatrix of a Hermitian matrix interlace those of the whole matrix (see \cite[Theorems 4.3.8 and 4.3.15]{HoJo13}).
This implies that the eigenvalues of any induced subdigraph interlace those of the digraph itself.

\begin{corollary}
\label{cor:interlacing digraph}
The $N^\a$-eigenvalues of an induced subdigraph interlace the $N^\a$-eigenvalues of the digraph.
\end{corollary}

To see a simple example how useful the interlacing theorem is, let us consider the following notion.
Let $\eta^+(D)$ denote the number of non-negative $N^\a$-eigenvalues of a digraph $D$ and $\eta^{-}(D)$ denote the number of non-positive $N^\a$-eigenvalues. The spectral bound of Cvetkovi\'{c} (see \cite{CDS95}) for the largest independent set of a graph extends to digraphs.
Here we will say that a vertex-set $S\subseteq V(D)$ is \emph{independent} if no two vertices in $S$ are joined by an arc in $D$.

\begin{proposition}
\label{prop:indep}
If $D$ has an independent set of size $k$, then for every $\alpha$, we have that $\eta^+(D) \geq k$ and $\eta^-(D) \geq k$.
\end{proposition}

\begin{proof} Let $\nu_1 \geq \cdots \geq \nu_n$ be the eigenvalues of $N^\a(D)$. By interlacing, we see that $\nu_k \geq 0$ and so $N^\a(D)$ has at least $k$ non-negative eigenvalues. Applying the same argument to $-N^\a(D)$ shows that there are at least $k$ non-positive eigenvalues as well.
\end{proof}

Since the independent sets in $D$ correspond to independent sets in the underlying undirected graph $G$, one can optimize the bound of Proposition \ref{prop:indep} over all orientations of $G$ and over different choices of $\alpha$. One can even add any real weights on the edges (which gives rise to a ``complex version" of the Lov\'asz $\vartheta$-function).

\subsection{Symmetry of the spectrum}

There is an essential difference when considering the symmetry of the spectra of the two kinds of Hermitian adjacency matrices. For the Hermitian matrix of the first kind from \cite{GuMo}, the spectrum of every digraph without loops or digons is symmetric about 0. This is no longer true for the Hermitian matrix of the second kind (see the example in Figure \ref{fig:digraph large negative}(a)). However, we still have the behaviour from the undirected case. Recall that a digraph $D$ is \emph{bipartite} if there is a \emph{bipartition} $V(D) = X \cup Y$ such that every arc in $D$ has one end in $X$ and the other end in $Y$.

\begin{theorem}
\label{thm:bipartite}
If D is bipartite, then the $M$-spectrum of $D$ is symmetric about 0, i.e., if $\nu$ is an eigenvalue of $M(D)$, then $-\nu$ is an eigenvalue of $M(D)$ with the same multiplicity.
\end{theorem}

\begin{proof}
The theorem is a simple consequence of the following fact: If $Mz=\nu z$, let $z'$ by the vector which agrees with $z$ on $X$ and agrees with $-z$ on $Y$, where $X\cup Y$ is the bipartition of $D$. Then $Mz' = -\nu z'$. The details are left to the reader.
\end{proof}

An alternative proof of Theorem \ref{thm:bipartite} can be given by looking at the coefficients of the characteristic polynomial $\phi(M(D),\nu) = \sum_{k=0}^n c_k \nu^{n-k}$. These coefficients have combinatorial interpretation via expansion of the determinant as the sum over all permutations of $V(D)$. If $D$ is bipartite, there are no cycles of odd length, thus all coefficients $c_k$ with $k$ odd are zero, and hence the eigenvalues are symmetric with respect to 0.

\begin{figure}
\centering
\includegraphics[scale=1.0]{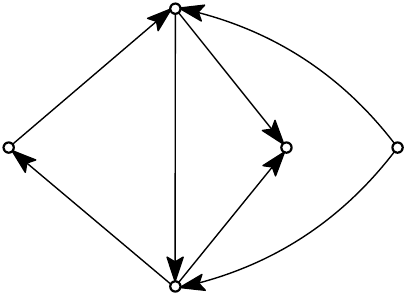}
\caption{A nonbipartite graph whose characteristic polynomial is $\phi(N,\nu)=\nu^5-7\nu^3+6\nu$ and its spectrum is symmetric about 0.}\label{fig:symmetricspectrum}
\end{figure}

There are nonbipartite digraphs whose $M$-spectrum is symmetric about 0. An example is given in Figure \ref{fig:symmetricspectrum}.

\section{Spectral radius}
\label{sec:rho}

Recall that the \emph{spectral radius} $\rho(Q)$ of a matrix $Q$ is defined as
\[ \rho(Q) = \max\{|\lambda| \mid \lambda \text{ an eigenvalue of }Q \}
\]
and we also define $\rho(D) = \rho(M(D))$ as the \emph{spectral radius} of the digraph $D$.

Let us start with a simple observation.

\begin{proposition}
\label{prop:average degree}
Suppose that a mixed graph $G$ with $n$ vertices has $s$ arcs and $t$ undirected edges. Let $d = \tfrac{s+2t}{n}$. Then
$\mu_1(G) \ge d$.
\end{proposition}

\begin{proof}
To prove the inequality, we just take the constant real vector $x$ with coordinates $x_v=n^{-1/2}$. Clearly, $\Vert x\Vert = 1$. By (\ref{eq:zNz}), it is easy to see that $x^*Mx = d$ (since $y=0$). This implies that $\mu_1(G)\ge d$.
\end{proof}

\begin{figure}
\centering
\includegraphics[scale=1.2]{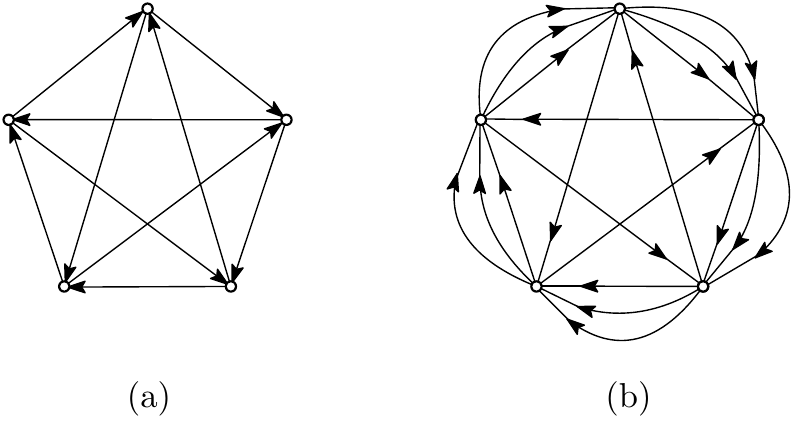}
\caption{(a) A simple circulant $Z_5$ with $\mu_1(Z_5) \approx 2.165$ and $\mu_5(Z_5) \approx -3.165$. (b) A circulant $C$ with multiple arcs has $\mu_1(C) \approx 4.0418$ and $\mu_5(C) \approx -6.8195$.}\label{fig:digraph large negative}
\end{figure}

For a digraph $D$, let the eigenvalues of $N^\a = N^\a(D)$ be $\nu_1 \geq \cdots \geq \nu_n$. Since $N^\a$ is not a matrix with non-negative real entries, there is no analogue of the Perron value of the adjacency matrix and the properties of $\nu_1$ may be a bit unintuitive. In particular, it may happen that $\nu_1 < |\nu_n| = \rho(N^\a)$. For example, the digraph shown in Figure \ref{fig:digraph large negative}(b) has $|\mu_5|/\mu_1 > 1.68$. The purpose of this section is to show that $|\nu_n|$ is still bounded in terms of $\nu_1$.

We will first give a general upper bound that holds for arbitrary Hermitian matrices $N^\a(D)$ as long as the real part of $\a$ is non-negative.

\begin{theorem}
\label{thm:spectral_radius}
Let $D$ be a digraph and $\a = a+bi$, where $a,b\in \RR$ and $a\ge0$. Then
$$
   \tfrac{1}{3}\, \rho(N^\a(D)) \le \nu_1(N^\a(D)) \le \rho(N^\a(D)).
$$
\end{theorem}

\begin{proof}
While the second inequality follows from the definition of the spectral radius, the first one needs a proof. We may assume that $|\a|=1$. If $\rho(N^\a)\ne \nu_1$, then $\rho = \rho(N^\a) = |\nu_n|$. Let $z=x+iy$ be a unit eigenvector for $\nu_n$. We may assume w.l.o.g. that $x\ne 0$ and $y\ne 0$. By (\ref{eq:zNz}), we have
\begin{equation}\label{eq:XYZ}
   -\rho = \nu_n = z^*N^\a z = X + Y + Z,
\end{equation}
where
$$
   X = 2a \sumvtou{x_vx_u}, \quad Y = 2a \sumvtou{y_vy_u}, \quad \textrm{and} \quad Z = 2b \sumvtou{(y_vx_u - x_vy_u)}.
$$

\newcommand\hatx{\hat x}
\newcommand\haty{\hat y}

Suppose first that $|X|+|Y|\ge \rho/3$. Let ${\hatx},{\haty}\in\RR^V$ be real vectors whose coordinates are $|x_v|$ and $|y_v|$, respectively. Then ${\hatx}^* N^\a {\hatx} \ge |X|$ and ${\haty}^* N^\a {\haty} \ge |Y|$. Since $\Vert {\hatx}\Vert^2 + \Vert {\haty}\Vert^2 = \Vert z\Vert^2 = 1$, it follows that
${\hatx}^* N^\a {\hatx} / \Vert {\hatx}\Vert^2 \ge |X|+|Y|$ or ${\haty}^* N^\a {\haty} / \Vert {\haty}\Vert^2 \ge |X|+|Y|$. In either case, it follows that $\nu_1\ge |X|+|Y|\ge \rho/3.$

Suppose now that $|X|+|Y|\le \rho/3$. Since $\rho = -X-Y-Z \le |X|+|Y| - Z$, we conclude that $\tfrac{2}{3}\,\rho \le -Z$.
Let us consider $\overline z = x -iy$. It follows that
$$
   \nu_1 \ge {\overline z}^* N^\a {\overline z} = X + Y - Z.
$$
By adding (\ref{eq:XYZ}) to this inequality, we obtain that
$$
   \rho = -X-Y-Z \le -X-Y + \nu_1 -X - Y \le 2(|X|+|Y|) + \nu_1 \le \tfrac{2}{3}\,\rho + \nu_1,
$$
which implies that $\rho \le 3\nu_1$.
\end{proof}

Theorem \ref{thm:spectral_radius} is very similar to the corresponding theorem in \cite[Theorem 5.6]{GuMo} and gives the same constant $\tfrac{1}{3}$ in the lower bound. However, the proof here is essentially different and gives a new proof of \cite[Theorem 5.6]{GuMo}. On the other hand, the proof of \cite[Theorem 5.6]{GuMo} does not work for our Theorem \ref{thm:spectral_radius}.

While the factor $\tfrac{1}{3}$ in Theorem \ref{thm:spectral_radius} is tight for Hermitian matrices of the first kind \cite{GuMo}, it is not tight for the second kind, where it can be strengthened as follows.

\begin{theorem}
\label{thm:spectral_radius_new}
If $D$ is a digraph and $M=M(D)$ is its Hermitian matrix of the second kind, then
$$
   \tfrac{1}{2} \rho(M) \le \mu_1(D) \le \rho(M).
$$
\end{theorem}

\begin{proof}
We may assume that $\rho = \rho(M) = -\mu_n > \mu_1$.
Let $z\in \CC^V$ be a unit eigenvector for $\mu_n$ and let $\widehat{z}\in \RR^V$ be the vector with entries $\widehat{z}_v = |z_v|$, $v\in V$. Let us observe that $\Vert \widehat{z} \Vert = 1$ and that the following holds:
\begin{equation}\label{eq:half}
   \widehat{z}_v \widehat{z}_u = |\omega|\, |z_v|\, |z_u| = |\omega \bar{z}_v z_u| =
   \tfrac{1}{2}\,|\omega \bar{z}_v z_u| + \tfrac{1}{2}\,|\bar\omega \bar{z}_u z_v| \ge
   \tfrac{1}{2}\,|\omega \bar{z}_v z_u + \bar\omega \bar{z}_u z_v|.
\end{equation}
By using this inequality and (\ref{eq:zNz}), we obtain the following:
$$
  \tfrac{1}{2}\,\rho = \tfrac{1}{2}\,\left| z^* M z \right| = \tfrac{1}{2}\,\left| \sumvtou (\omega\bar{z}_v z_u + \bar\omega \bar{z}_u z_v)\right| \le
  \sumvtou \widehat{z}_v \widehat{z}_u = \widehat{z}^* M \widehat{z} \le \mu_1.
$$
This completes the proof.
\end{proof}

The factor $\tfrac{1}{2}$ in Theorem \ref{thm:spectral_radius_new} is best possible. This is justified with the example of the directed cycle of length 3, whose $M$-eigenvalues are $1$ (with multiplicity 2) and $-2$. By using the Cartesian product operation (see \cite{GuMo}), one can boost this example to obtain bigger graphs and larger spectral radii. The example of the 3-cycle shows another ``anomaly" that the largest eigenvalue of a (strongly) connected digraph need not be simple.

The above proof gives the following improvement for any $\alpha=a+bi$, where $|\a|=1$ with $a > \tfrac{1}{3}$:
$$
   a \, \rho(N^\a(D)) \le \nu_1(N^\a(D)) \le \rho(N^\a(D)).
$$

\subsection*{Acknowledgement}
The author is grateful to an anonymous referee whose insight provided a strengthening of Theorem \ref{thm:spectral_radius} that is now given in Theorem \ref{thm:spectral_radius_new}.

\bibliographystyle{abbrv}
\bibliography{NewHermitian}

\end{document}